\documentclass[12pt,a4paper]{amsart}
\usepackage{amsmath, amssymb}
\usepackage{hyperref}
\usepackage[english]{babel}
\theoremstyle{plain}

\usepackage{cite}

\newtheorem{Theorem}{Theorem}
\newtheorem{Lemma}{Lemma}

\newtheorem{Corollary}{Corollary}

\newtheorem{Conjecture}{Conjecture}

\begin{document}
	
	\title{Derivations of Lie Algebras of Vector Fields in Infinitely Many Variables}

		\author
	{Oksana Bezushchak}
	\address{Oksana Bezushchak: Faculty of Mechanics and Mathematics, Taras Shevchenko National University of Kyiv, Volodymyrska, 60, Kyiv 01033, Ukraine}
	\email{bezushchak@knu.ua}
	
	\author{Iryna Kashuba}
	\address{Iryna Kashuba: Shenzhen International Center for Mathematics, Southern University of Science and Technology, 1088 Xueyuan Avenue, Nanshan District, Shenzhen, Guangdong, China}
	\email{kashuba@sustech.edu.cn}    
	
	\maketitle
	
		\begin{abstract}
	Let $W_X(\mathbb{F})$ be the Lie algebra of all derivations of the polynomial algebra $\mathbb{F}[X]$ in infinitely many variables.	We describe  all derivations of $W_X(\mathbb{F})$  over a field of  characteristic zero and  prove that all  such derivations are inner. We also consider   the subalgebras $W_\infty(\mathbb{F})$  and $W_{\text{fin}}(\mathbb{F})$ of the algebra $W_X(\mathbb{F})$ and describe  all of their derivations. 	
	\end{abstract}

	\section{Introduction} 
	
	Let $\mathbb{F}$ be a field of zero characteristic, and let $\mathbb{F}[X_n]$, $n \geq 1$, be the polynomial algebra on the set of variables $X_n = \{x_1, \dots, x_n\}$.
	
	The Lie algebra:
	\[
	W_n(\mathbb{F}) = \text{Der } \mathbb{F}[X_n] = \left\{ \sum_{i=1}^{n} f_i(x_1, \dots, x_n) \frac{\partial}{\partial x_i} \ \Big| \ f_i \in \mathbb{F}[X_n] \right\}
	\]
	of all derivations of the algebra $\mathbb{F}[X_n]$ is an infinite-dimensional simple Lie algebra of Cartan type (see \cite{Rudakov1,Strade}).
	
	Consider the ascending chains: \[
	X_1 \subset X_2 \subset \dots, \quad \text{and} \quad 
	W_1(\mathbb{F}) \subset W_2(\mathbb{F}) \subset \dots .
	\]
	Let $$X = \bigcup_{n \geq 1} X_n ,\quad \text{and} \quad  W_\infty(\mathbb{F}) = \bigcup_{n \geq 1} W_n(\mathbb{F}).$$ 
	
	The Lie algebra $W_\infty(\mathbb{F})$ is a subalgebra of the Lie algebra $W_X(\mathbb{F})$ of all derivations of the polynomial algebra $\mathbb{F}[X]$ in infinitely many variables:
	\[
	W_X(\mathbb{F}) = \text{Der } \mathbb{F}[X] = \left\{ \sum_{i=1}^{\infty} f_i(x_1, \ldots ) \frac{\partial}{\partial x_i} \ \Big| \ f_i \in \mathbb{F}[X] \right\}.
	\] We will consider also the Lie algebra $W_{\text{fin}}(\mathbb{F})$ lying strictly between $W_\infty(\mathbb{F})$ and $W_X(\mathbb{F})$. The algebra $W_{\text{fin}}(\mathbb{F})$ consists of derivations
	\[
	\sum_{i=1}^{\infty} f_i(x_1, \ldots ) \frac{\partial}{\partial x_i}
	\]
	such that each variable $x_n \in X$ occurs in only finitely many polynomials $f_i$, $i \geq 1$. It is easy to see that $W_\infty(\mathbb{F})$ is an ideal of the algebra $W_{\text{fin}}(\mathbb{F})$.

	Floris Takens \cite{Takens} (se also \cite{Farnsteiner,Morimoto}) proved that every derivation of the Lie algebra  $W_n(\mathbb{F})$ over a field of characteristic $0$ is inner. Richard E. Block, Robert L. Wilson, Helmut Strade, and Alexander Premet~\cite{Strade} studied $W_n(\mathbb{F})$ in positive characteristic. For more recent papers concerning algebraic structure of $W_n(\mathbb{F})$, see 	\cite{Bezushchak_Petravchuk_Zelmanov,Klymenko_Lysenko_Petravchuk,Makedonskyi_Petravchuk}.

	In this paper, we describe derivations of the algebra $W_\infty(\mathbb{F}),$ $W_{\text{fin}}(\mathbb{F}),$ and $ W_X(\mathbb{F}) $ over a field $\mathbb{F}$  of zero characteristic.
	
	\begin{Theorem}\label{Theo1} Every derivation of the algebra $W_\infty(\mathbb{F})$ is a restriction of an adjoint operator $\text{ad}(a):x\to [x,a]$, where $a \in W_{\text{fin}}(\mathbb{F})$. \end{Theorem}
	
	\begin{Theorem}\label{Theo2}  All derivations of the algebra $W_{\text{fin}}(\mathbb{F})$ are inner. \end{Theorem}

	\begin{Theorem}\label{Theo3} All  derivations of algebra $W_X(\mathbb{F})$ are inner. \end{Theorem}

	From Theorem \ref{Theo1}, it follows that not all derivations of $W_{\infty}(\mathbb{F})$ are inner, but the ideal of inner derivations is dense in the algebra of all derivations of $W_{\infty}(\mathbb{F})$ in the Tychonoff topology. Dragomir Ž. Doković and Kaiming Zhao  \cite{Doković_Zhao} proved that the ideal of inner derivations is dense in $W_X(\mathbb{F})$ in the Tychonoff topology in a more general context.

	There is an analogy between polynomial algebras in infinitely many variables and algebras of infinite matrices. Let $\mathbb{N}=\{1,2,\ldots\}$ be the set of all positive integers. We consider infinite $\mathbb{N} \times \mathbb{N}$ matrices over the field $\mathbb{F}$ of arbitrary characteristic. The Lie algebra $\mathfrak{gl}_\infty(\mathbb{F})$ consists of $\mathbb{N} \times \mathbb{N}$ matrices having finitely many nonzero elements, and let $$\mathfrak{sl}_\infty(\mathbb{F}) = [\mathfrak{gl}_\infty(\mathbb{F}), \mathfrak{gl}_\infty(\mathbb{F})].$$ Clearly, $$\mathfrak{sl}_\infty(\mathbb{F}) = \bigcup_{n=2}^{\infty} \mathfrak{sl}_n(\mathbb{F}).$$ The Lie algebra $\mathfrak{gl}_{\mathbb{N}}(\mathbb{F})$ consists of $\mathbb{N} \times \mathbb{N}$ matrices having finitely many nonzero entries in each column. Finally, the intermediate algebra $$\mathfrak{gl}_{\text{fin}}(\mathbb{F}), \quad  \mathfrak{gl}_{\infty}(\mathbb{F}) \subseteq \mathfrak{gl}_{\text{fin}}(\mathbb{F}) \subset \mathfrak{gl}_{\mathbb{N}}(\mathbb{F}),$$ consists of $\mathbb{N} \times \mathbb{N}$ matrices having finitely many nonzero entries in each row and in each column.

	C.-H. Neeb \cite{Neeb} showed that an arbitrary derivation of the Lie algebra $\mathfrak{sl}_\infty(\mathbb{F})$ is of the type $\text{ad}(a)$, $a \in \mathfrak{gl}_{\text{fin}}(\mathbb{F})$, provided that $\text{char} \, \mathbb{F} = 0$. In~\cite{Bezushchak2,Bezushchak_Kashuba_Zelmanov}, we extended this result to algebras of infinite matrices over fields of positive characteristics. We also proved that all derivations of the algebras $\mathfrak{gl}_{\text{fin}}(\mathbb{F})$ and $\mathfrak{gl}_N(\mathbb{F})$ are inner.
	
	The approach in \cite{Bezushchak2} was as follows:
	(1) to lift derivations of Lie algebras of infinite matrices to derivations of associative algebras of infinite matrices, using the solution to Herstein's problems by K.~Beidar, M.~Bre\v sar, M.~Chebotar, and W.S.~Martindale 3nd  (see, \cite{Bei_Bre_Cheb_Mart_1,Bei_Bre_Cheb_Mart_2,Bei_Bre_Cheb_Mart_3});
	(2)~to use N.~Jacobson's description of derivations of dense algebras of linear transformations \cite{Jacobson}.
	
	This approach does not work for Lie algebras of derivations of polynomials. For this reason, we will follow a different path, one that is closer to the description of derivations of locally matrix algebras; see~\cite{Bezushchak1,BezOl_2}.
	
	\section{On locally finite-dimensional $\mathfrak{sl}_\infty(\mathbb{F})$-modules}  
	
Let $\mathbb{N}$ be the set of positive integers.	Consider the Lie algebra:
	\[
	\text{span} \left\{ x_i \frac{\partial}{\partial x_j}   , \ i\not= j; \quad x_i \frac{\partial}{\partial x_i} - x_j \frac{\partial}{\partial x_j}, \ i, j \in \mathbb{N} \right\}.
	\]
	This algebra is isomorphic to $\mathfrak{sl}_\infty(\mathbb{F})$. Hence, we can view $\mathfrak{sl}_\infty(\mathbb{F})$ as a subalgebra of $W_\infty(\mathbb{F})$.  If we consider the Lie algebra \[
	\text{span} \left\{ x_i \frac{\partial}{\partial x_j}   ,  \ i, j \in \mathbb{N} \right\}
	\] that is isomorphic to $\mathfrak{gl}_\infty(\mathbb{F})$, then  the subalgebra $\mathfrak{sl}_\infty(\mathbb{F})$  is an ideal of the algebra $\mathfrak{gl}_\infty(\mathbb{F})$.

	Let $$X^* = \bigcup_{n= 0}^{\infty} \underbrace{X \cdots X}_n$$ be the set of all monomials in $X$; and let
	\[
	\widetilde{\mathbb{F}[X]} = \left\{ \sum_{m \in X^{*}} \alpha_m m \ \Big| \ \alpha_m \in \mathbb{F} \right\}
	\]
	be the vector space of infinite sums that can be identified with the vector space of all maps $X^* \to \mathbb{F}$. This vector space is in fact an algebra since every monomial can be expressed as a product of monomials in finitely many ways.
	
	Consider the vector space:
	\begin{equation}\label{eq1}
		\widetilde{W} = \sum_{i=1}^{\infty} \widetilde{\mathbb{F}[X]} \frac{\partial}{\partial x_i}.
	\end{equation}
	Infinite sums \eqref{eq1} cannot always be identified with derivations of the algebra $\widetilde{\mathbb{F}[X]}$, and a commutator of sums \eqref{eq1}  may not make sense. However, $\widetilde{W}$ can be viewed as a $W_\infty(\mathbb{F})$-module and, consequently, also as $\mathfrak{sl}_\infty(\mathbb{F})$-module, since for any $h_i\in \mathbb{F}[X]$ and $f_j \in \widetilde{\mathbb{F}[X]}$, we have:
	\[
	\Big[\sum_{i=1}^{n} h_i \frac{\partial}{\partial x_i}, \sum_{j=1}^{\infty} f_j \frac{\partial}{\partial x_j}\Big] = \sum_{i=1}^{n}   \Big[h_i \frac{\partial}{\partial x_i }, \sum_{j=1}^{\infty} f_j \frac{\partial }{\partial x_j}\Big] , \quad \text{where}\] 
	\[   \Big[h_i \frac{\partial}{\partial x_i}, \sum_{j=1}^{\infty} f_j \frac{\partial}{\partial x_j}\Big]= \sum_{j=1}^{\infty} h_i \Big(\frac{\partial f_j}{\partial x_i}\Big) \frac{\partial}{\partial x_j} -  \sum_{j=1}^{\infty} f_j \Big(\frac{\partial h_i}{\partial x_j}\Big) \frac{\partial}{\partial x_i} \in  \widetilde{W} ,\] i.e., $[ W_\infty(\mathbb{F}),\widetilde{W}] \subseteq \widetilde{W}.$ 
	
	Recall that a module $M$ over a Lie algebra $L$ is called \textit{locally finite-dimensional} if given a finite subset $S \subset M$ and a finitely generated subalgebra $L_1 \subset L$, the subset $S$ generates a finite-dimensional module over $L_1$.

	\begin{Lemma}\label{Lemma1} Let $L$ be a finite-dimensional Lie algebra over a field $\mathbb{F}$. Let $M_1, \dots, M_s$ be finite-dimensional $L$-modules, $s \geq 1$. Let $\{V_i: i \in I\}$ be a collection of $L$-modules, for some set of indices $I$, where each $V_i$ is isomorphic to one of the modules $M_1, \dots, M_s$. Then the Cartesian sum
		\[
		V = \bigoplus_{i \in I} V_i
		\]
		is a locally finite-dimensional module over $L$. \end{Lemma}

	\begin{proof} Let $a \in L$. Consider the linear transformations
		\[
		a\big|_{M_i} \in \text{End}_{\mathbb{F}}(M_i), \quad a\big|_{M_i}: x \mapsto ax, \quad x \in M_i.
	 \quad \text{for} \quad 1 \leq i \leq s.	\] Let $f_i(t)$ be the characteristic polynomial of the linear transformation $a|_{M_i}$. Then
		\[
		f_i(a|_{M_i}) = 0.
		\]
		
		Now set \( f(t) = f_1(t)\cdots f_s(t) \). Then $f(a\big|_{V_i}) = 0$ for any $i \in I$, and therefore $f(a\big|_V) = 0$. By the Poincaré–Birkhoff–Witt Theorem (see~\cite{Jacobson_PBW}), it implies that the module $V$ is locally finite-dimensional. This completes the proof of the lemma. \end{proof}
	
	For a monomial $m \in X^*$ of length $\kappa$, we define the degree $$\deg\Big(m \frac{\partial}{\partial x_i}\Big) = \kappa - 1.$$ Let $\widetilde{W}_\kappa$ denote the subspace $\widetilde{W}_\kappa \subset \widetilde{W}$ of elements of degree $\kappa$, $-1 \leq \kappa < \infty$. Clearly, $\widetilde{W}_\kappa$ is a $\mathfrak{sl}_{\infty}(\mathbb{F})$-submodule.
	
	\begin{Lemma}\label{Lemma2} For any $n \geq 1$, $-1 \leq \kappa < \infty$, the $\mathfrak{sl}_n(\mathbb{F})$-module $\widetilde{W}_\kappa$ is locally finite-dimensional. \end{Lemma}

	\begin{proof} Let $\mathbb{F}[X_n]_i$ be the linear span of all monomials in $X_n$ of length $i$; and let $U_i$ be the linear span of all monomials in $X \setminus X_n$ of length $i$ for $i\geq 1$. Then:
		\[
		\widetilde{W}_k = \sum_{1\leq j\leq n, \ 0 \leq i \leq k+1} \mathbb{F}[X_n]_i \ U_{k-i+1} \frac{\partial}{\partial x_j} + \sum_{j>n, \ 0 \leq i \leq k+1} \mathbb{F}[X_n]_i \ U_{k-i+1} \frac{\partial}{\partial x_j}
		\] for $-1 \leq k < \infty .$ This is a Cartesian sum of finite-dimensional $\mathfrak{sl}_n(\mathbb{F})$-modules,  $$\text{each isomorphic to} \quad \sum_{j=1}^{n} \ \mathbb{F}[X_n]_i \ \frac{\partial}{\partial x_j} \quad \text{or to} \quad \mathbb{F}[X_n]_i, \quad 0 \leq i \leq k+1.$$ By Lemma \ref{Lemma1}, the $\mathfrak{sl}_n(\mathbb{F})$-module $\widetilde{W}_k$ is locally finite-dimensional. This completes the proof of the lemma. \end{proof}
	
	The following consequence follows from Lemma \ref{Lemma2}.
	
	\begin{Corollary}\label{cor1} For  $-1  \leq k < \infty$, the $\mathfrak{sl}_\infty(\mathbb{F})$-module $\widetilde{W}_\kappa$ is locally finite-dimensional. \end{Corollary}
	
	Consider an element $w \in \widetilde{W}$,
	\[
	w = \sum_{l=1}^{\infty} f_l \frac{\partial}{\partial x_l}, \quad \text{where} \quad f_l \in \widetilde{\mathbb{F}[X]}.
	\]
	Let $1 \leq i \not= j < \infty$. Then
	\begin{equation}\label{eq2}
		\Big[w, x_j \frac{\partial}{\partial x_i}\Big] = f_j \frac{\partial}{\partial x_i} - \sum_{l=1}^{\infty} x_j \left(\frac{\partial f_l}{\partial x_i} \right) \frac{\partial}{\partial x_l}.
	\end{equation}
	
	The $\frac{\partial}{\partial x_i}$-component of the commutator above is:
	\begin{equation}\label{eq3}
		f_j - x_j \frac{\partial f_i}{\partial x_i}.
	\end{equation}
	
	Furthermore, we have:
	\[
	\Big[w, x_i \frac{\partial}{\partial x_i}\Big] = f_i \frac{\partial}{\partial x_i} - \sum_{l=1}^{\infty} \left(\frac{\partial f_l}{\partial x_i} \right) x_i \frac{\partial}{\partial x_l}, \quad \text{and}
	\]
	\[
	\Big[w, x_j \frac{\partial}{\partial x_j}\Big] = f_j \frac{\partial}{\partial x_j} - \sum_{l=1}^{\infty} \left(\frac{\partial f_l}{\partial x_j} \right) x_j \frac{\partial}{\partial x_l}.
	\] Hence:
	\begin{equation}\label{eq4}
		\Big[w, x_i \frac{\partial}{\partial x_i} - x_j \frac{\partial}{\partial x_j}\Big] = f_i \frac{\partial}{\partial x_i} - f_j \frac{\partial}{\partial x_j} + \sum_{l=1}^{\infty} \left(\frac{\partial f_l}{\partial x_j} x_j - \frac{\partial f_l}{\partial x_i} x_i \right) \frac{\partial}{\partial x_l}.
	\end{equation}

	\begin{Lemma}\label{Lemma3}  Let $n\geq 3.$ Then $$\Big\{w \in \widetilde{W} \ \Big|\  \mathfrak{sl}_n(\mathbb{F}) w = (0) \Big\} = $$ $$\widetilde{\mathbb{F}[X \setminus X_n]} \ \Bigg(\sum_{i=1}^{n} x_i \frac{\partial}{\partial x_i}\Bigg) + \sum_{l=n+1}^{\infty} \widetilde{\mathbb{F}[X \setminus X_n]} \frac{\partial}{\partial x_l}.$$ \end{Lemma}

	\begin{proof} Let 	\[
		w = \sum_{l=1}^{\infty} f_l \frac{\partial}{\partial x_l}, \quad \text{where} \quad f_l \in \widetilde{\mathbb{F}[X]} \quad \text{and} \quad \mathfrak{sl}_n(\mathbb{F}) w = (0) .
		\] From \eqref{eq2}, it follows that for any $l$ and $i,$ $l \geq n+1,$ $1 \leq i \leq n, $ we have
		\[
		\frac{\partial f_l}{\partial x_i} = 0, \quad  \text{ so } \quad f_l \in \mathbb{F}[X \setminus X_n].
		\]
		Thus, it follows that:
		\begin{equation}\label{eq4+}
			\mathfrak{sl}_n(\mathbb{F}) \ \Big(\sum_{i=1}^{n} f_i \frac{\partial}{\partial x_i} \Big) = (0).\end{equation}
		
		Let $1\leq i, j, l \leq n$ be distinct integers. By \eqref{eq4}, we have $$ \frac{\partial f_l}{\partial x_j} x_j =\frac{\partial f_l}{\partial x_i} x_i .$$

		For a monomial $m$, let $\deg_{x_i}(m)$ be the number of occurrences of $x_i$ in $m$. We showed that for every monomial $m$ appearing in $f_l$, and for any $1 \leq i \neq j \leq n$ with  $i, j$ distinct from $l$, the following holds:
		\[
		\deg_{x_i}(m) = \deg_{x_j}(m).
		\]
		It means that $$m = x_l^p (x_i \cdots x_{l-1} x_{l+1}  \cdots x_n)^q m',$$ where $m'$ is a monomial in $X \setminus X_n$. From the equality \eqref{eq2}, it follows that $q = 0$, $m = x_l^p m'$.
		
		The equality \eqref{eq3} implies that $f_i = x_i a_i$, where $a_i \in \widetilde{\mathbb{F}[X \setminus X_n]}$ for $1 \leq i \leq n$, and $a_i = a_j$ for all $1 \leq i, j \leq n$. This completes the proof of the lemma. \end{proof}
	
	The following consequence follows from Lemma \ref{Lemma3}.
	
	\begin{Corollary}\label{cor2} We have: $$ \big\{ w\in \widetilde{W} \ \big| \ \mathfrak{sl}_{\infty}(\mathbb{F}) w =(0) \big\}= \mathbb{F}\Big(\sum_{i=1}^{\infty} x_i \frac{\partial}{\partial x_i}\Big).$$ \end{Corollary}

	\begin{Lemma}\label{Lemma4} For any $n \geq 2$ and $-1 \leq k < \infty$, the first cohomology group vanishes:
		\[
		H^1\big(\mathfrak{sl}_n(\mathbb{F}), \widetilde{W}_k\big) = 0.
		\] \end{Lemma}

	\begin{proof}  Let $d: \mathfrak{sl}_n(\mathbb{F}) \to \widetilde{W}_k$ be a derivation; and  let $V$ be the $\mathfrak{sl}_n(\mathbb{F})$-submodule generated by $d(\mathfrak{sl}_n(\mathbb{F}))$. Since $\mathfrak{sl}_n(\mathbb{F})$ is perfect, i.e.,  $\mathfrak{sl}_n(\mathbb{F}) = [\mathfrak{sl}_n(\mathbb{F}) , \mathfrak{sl}_n(\mathbb{F})]$, it follows that $d(\mathfrak{sl}_n(\mathbb{F})) \subseteq V$. Lemma \ref{Lemma2} implies that $\dim_{\mathbb{F}} V < \infty$. Hence, the derivation $d: \mathfrak{sl}_n(\mathbb{F}) \to V$ is inner; see~\cite{Jacobson_PBW}. This completes the proof of the lemma. \end{proof}
	
	\begin{Lemma}\label{Lemma5} For any $-1 \leq k < \infty$, the first cohomology group vanishes: $$H^1\big(\mathfrak{sl}_{\infty}(\mathbb{F}), \widetilde{W}_k\big) = (0).$$\end{Lemma}

	\begin{proof} For an element $w \in \widetilde{W}_k$, let $\hat{w}$ denote the derivation defined by $$\hat{w}: x \mapsto xw, \quad  x \in \mathfrak{sl}_\infty(\mathbb{F}).$$ By Lemma \ref{Lemma4}, for any $n \geq 2$, there exists an element $w_n \in \widetilde{W}_k$ such that:
		\[
		(d - \hat{w}_n)(\mathfrak{sl}_n(\mathbb{F})) = (0).
		\]
		
		For an element $w \in \widetilde{W}$, a monomial $m\in X^{*}$ and   an integer $i\geq 1$, let $\gamma_w(m \frac{\partial}{\partial x_i})$ be the coefficient of $m \frac{\partial}{\partial x_i}$ in $w$. We have:
		\[
		\mathfrak{sl}_n(\mathbb{F})\left( \sum_{i=1}^{n} x_i \frac{\partial}{\partial x_i} \right) = (0);
		\] see \eqref{eq4+}.
		Considering elements $$w_n - \gamma_{w_n} \Big(x_1 \frac{\partial}{\partial x_1}\Big) \left( \sum_{i=1}^{n} x_i \frac{\partial}{\partial x_i} \right)$$ instead of $w_n$, we will assume that $$\gamma_{w_{n}}\Big(x_1 \frac{\partial}{\partial x_1}\Big) = 0 \quad \text{for any} \quad n \geq 2.$$
		We claim that for every monomial $m \in X^*$ and every integer $i\geq 1$, the sequence \begin{equation}\label{eq4++} \gamma_{w_2}\Big(m \frac{\partial}{\partial x_i}\Big), \quad \gamma_{w_3}\Big(m \frac{\partial}{\partial x_{i}}\Big), \quad \dots \end{equation} stabilizes. Indeed, let $m \in X^*_n$ and suppose $i<n$. Then $$\mathfrak{sl}_n(\mathbb{F}) \left( w_n - w_{n+1} \right) = (0).$$ Therefore, by Lemma \ref{Lemma3}, we have:
		\[
		w_n - w_{n+1} \in \widetilde{\mathbb{F}[X \setminus X_n]} \Big(\sum_{j=1}^{n} x_j \frac{\partial}{\partial x_j} \Big)+ \sum_{l=n+1}^{\infty} \widetilde{\mathbb{F}[X \setminus X_n]} \frac{\partial}{\partial x_l}.
		\]
		The term $m \frac{\partial}{\partial x_i}$ cannot appear on the right-hand side unless $m = x_i$. Hence, assuming $m \neq x_i$, we obtain::
		\[
		\gamma_{w_n}\Big(m \frac{\partial}{\partial x_i}\Big) = \gamma_{w_{n+1}}\Big(m \frac{\partial}{\partial x_i}\Big).
		\]
		Again, by Lemma \ref{Lemma3}:
		\[
		\gamma_{w_n -w_{n+1}}\Big(x_i \frac{\partial}{\partial x_i}\Big) = \gamma_{w_n -w_{n+1}}\Big(x_1 \frac{\partial}{\partial x_1}\Big)= 0,
		\]
		and sequence \eqref{eq4++} stabilizes.
		
		Define
		\[
		\gamma\left(m \frac{\partial}{\partial x_i}\right) = \lim_{n \to \infty} 
		\gamma_{w_n}\left(m \frac{\partial}{\partial x_i}\right).
		\]
		Consider the element
		\[
		w = \sum_{m \in X^*,\ i \geq 1} \gamma\left(m \frac{\partial}{\partial x_i}\right) 
		m \frac{\partial}{\partial x_i}.
		\]
		From the construction above, it follows that
		\[
		\mathfrak{sl}_n(\mathbb{F})(w - w_n) =( 0).
		\]
		Thus, for all \( a \in \mathfrak{sl}_n(\mathbb{F}) \), we have
		\[
		[a, w] = [a, w_n] = d(a).
		\] This completes the proof of the claim. \end{proof}
	
	\begin{Lemma}\label{Lemma6} The first cohomology group vanishes: $$H^1\big(\mathfrak{sl}_{\infty}(\mathbb{F}), \widetilde{W}\big) = (0).$$ \end{Lemma}
	\begin{proof} Let $d:\mathfrak{sl}_{\infty}(\mathbb{F})\to \widetilde{W}$ be a derivation. For an arbitrary element $a\in \mathfrak{sl}_{\infty}(\mathbb{F}),$ the image $d(a)$ is an (infinite) sum of homogeneous components: $d(a)=d_{-1}(a)+d_0(a)+\cdots , $ where $d_i(a)\in \widetilde{W}_i.$ It is easy to  see that for every integer $i \geq -1$, the mapping $a\to d_i(a)$ is a derivation from $\mathfrak{sl}_{\infty}(\mathbb{F})$ to $\widetilde{W}_i$.  By Lemma \ref{Lemma5}, there exists an element $w_i \in \widetilde{W}_i$ such that $d_i(a) = [a, w_i]$ for an arbitrary element $a \in \mathfrak{sl}_\infty(\mathbb{F})$. Thus, $$w = \sum_{i=-1}^{\infty} w_i, \quad d = \hat{w}.$$ This completes the proof of the lemma. \end{proof}
	
	\section{The poof of Theorem \ref{Theo1}}
	
	The Lie algebras $W_{\text{fin}}(\mathbb{F})$ and $W_X(\mathbb{F})$ can be naturally identified with subspaces of $\widetilde{W}$. We denote these subspaces as $\widetilde{W}_{\text{fin}}, \widetilde{W}_X$ respectively. Recall that
	\[
	\widetilde{W}_X = \left\{ \sum_{i=1}^{\infty} f_i \frac{\partial}{\partial x_i} \ \middle| \ f_i \text{ are polynomials in } X \right\}.
	\]

	\begin{Lemma}\label{Lemma7} Let $w \in \widetilde{W}$, and $\mathfrak{sl}_\infty(\mathbb{F}) w \subseteq \widetilde{W}_X$. Then
		\[
		w = a \sum_{i=1}^{\infty} x_i \frac{\partial}{\partial x_i} + w',
		\]
		where $w' \in \widetilde{W}_X$, and for an arbitrary $i$, the partial derivative $\frac{\partial a}{\partial x_i}$ is a polynomial. \end{Lemma}
	
	\begin{proof} Let $$w = \sum_{i=1}^{\infty} f_i \frac{\partial}{\partial x_i}, \quad \text{where} \quad f_i \in \widetilde{\mathbb{F}[X]}.$$ From \eqref{eq2}, it immediately follows that for any $1\leq   i \neq j < \infty$, the partial derivative $\frac{\partial f_i}{\partial x_j}$ is a polynomial.

		The $\frac{\partial}{\partial x_i}$-component of the right-hand side of \eqref{eq4} is the following: 
		\[
		f_i \frac{\partial}{\partial x_i} + \left( \frac{\partial f_i}{\partial x_j} x_j - \frac{\partial f_i}{\partial x_i} x_i \right) \frac{\partial}{\partial x_i}.
		\]
		Hence, 
		\[
		f_i - \frac{\partial f_i}{\partial x_i} x_i 
		\] is also a polynomial.
		
		Let $$f_i = x_i \cdot h_i +\bar{f}_i, \quad \text{where} \quad  \bar{f}_i \in \mathbb{F}[X \setminus \{x_i\}], \quad h_i \in \mathbb{F}[X].$$ Thus, $$ f_i - \frac{\partial f_i}{\partial x_i} x_i =   x_i \cdot h_i + \bar{f}_i -\Big(h_i + x_i \frac{\partial h_i}{\partial x_i}\Big)x_i = \bar{f}_i  - x_i^2 \ \frac{\partial h_i}{\partial x_i}$$ is a polynomial. It implies that $\bar{f}_i$ is also a polynomial. Let $h_i = h_i' + h_i''$, where $h_i'$ is the sum of monomials involving $x_i$, and 
		\[
		\frac{\partial h_i''}{\partial x_i} = 0.
		\] It follows from the above that $h_i'$ is also a polynomial. Then, summarizing, we get that:
		$f_i = x_i \cdot g_i + p_i,$  where $ p_i $  is a polynomial and $\frac{\partial g_i}{\partial x_i} = 0.$
		
		Remark that
		\[
		w - \sum_{i=1}^\infty x_i \ g_i \ \frac{\partial}{\partial x_i} \in \widetilde{W}_X.
		\]
		Therefore, consider $ \sum_{i=1}^\infty x_i \ g_i \ \frac{\partial}{\partial x_i} $ instead of \( w \), that is, from now on 
		\[
		w := \sum_{i=1}^\infty x_i g_i \frac{\partial}{\partial x_i}.
		\] It follows from \eqref{eq3} that 
		\[
		x_j g_j - x_
		j \frac{\partial (x_i g_i)}{\partial x_i} = x_j (g_j - g_i)
		\]
		is a polynomial.
		
		Let \( a = h_1 \). Then for an arbitrary \( i \), we have $ g_i = a + (\text{a polynomial}).$ Now
		\[
		w \in a \sum_{i=1}^\infty x_i \frac{\partial}{\partial x_i} + \widetilde{W}_X.
		\] Since $$  \frac{\partial (a x_j)}{\partial  x_i}, \quad  i \neq j,$$ is a polynomial, it follows that 
		$\frac{\partial a}{\partial x_i}$ is a polynomial. This completes the proof of the lemma. 
	\end{proof}

	Consider the subalgebra 
	\[
	L := \sum_{i=1}^\infty \mathbb{F} \frac{\partial}{\partial x_i} + \mathfrak{gl}_\infty(\mathbb{F})
	\]
	of  the algebra $ W_\infty(\mathbb{F}).$
	
	\begin{Lemma}\label{Lemma8} The first cohomology group vanishes: $$H^1(L, \widetilde{W}_X) = (0).$$ \end{Lemma}
	
	\begin{proof} Let \( d : L \to \widetilde{W}_X \) be a derivation. Let's show that it is inner. By Lemma \ref{Lemma6}, there exists an element \( w \in \widetilde{W} \) such that
		\[
		d(x) = [w, x] \quad \text{for any } x \in \mathfrak{sl}_\infty(\mathbb{F}).
		\] By Lemma \ref{Lemma7}, $$ w = a \sum_{i=1}^\infty x_i \frac{\partial}{\partial x_i} + w',$$ where $ w' \in \widetilde{W}_X , $ and  $\partial a / {\partial x_i}$ is a polynomial for every  $i. $

		Take \( 1 \leq i \not= p < \infty \), $p\not= q.$ Then
		\[
		\left[ \frac{\partial}{\partial x_i}, x_p \frac{\partial}{\partial x_q} \right] = 0.
		\] Applying the derivation \( d \), we get:
		\[
		\left[ d\left( \frac{\partial}{\partial x_i}\right), \ x_p \frac{\partial}{\partial x_q} \right]
		+ \left[ \frac{\partial}{\partial x_i}, \Big[w, x_p \frac{\partial}{\partial x_q}\Big] \right]
		= \] \[\left[ d\left( \frac{\partial}{\partial x_i} \right) + \Big[\frac{\partial}{\partial x_i}, w\Big], \  x_p \frac{\partial}{\partial x_q} \right] = 0.
		\]
		
		Let
		\[
		d\left( \frac{\partial}{\partial x_i} \right) = \sum_{k=1}^\infty u_k \frac{\partial}{\partial x_k},\] where  $ u_k$  are polynomials.  We obtain: 
		\[
		\left[ \frac{\partial}{\partial x_i}, w \right] = \left[ \frac{\partial}{\partial x_i}, \ a\sum_{j=1}^\infty x_j \frac{\partial}{\partial x_j} +   w' \right]=
		\]
		\[ \left( \frac{\partial a}{\partial x_i}\right)  \sum_{j=1}^\infty x_j \frac{\partial}{\partial x_j} + a  \frac{\partial }{\partial x_i} + \left[ \frac{\partial}{\partial x_i}, w' \right].  \]

		Let
		\[
		\left[ \frac{\partial}{\partial x_i}, w' \right] = \sum_k v_k \frac{\partial}{\partial x_k}, \] where   $v_k$ are polynomials. The \( \frac{\partial}{\partial x_i} \)-component of the element \[   d\left(\frac{\partial}{\partial x_i}\right) +\left[  \frac{\partial}{\partial x_i}, w\right] \quad \text{is} \quad 
		\left(u_i + \left(\frac{\partial a}{\partial x_i}\right) x_i + a + v_i\right) \frac{\partial}{\partial x_i}.
		\]

		Suppose that the element $a$ is not a polynomial. Since \[  u_i + \left(\frac{\partial a}{\partial x_i} \right)x_i + v_i \] is a polynomial and every variable is involved in finitely many monomials in \( a \), it follows that there exists \( q \neq i \) such that
		\[
		\frac{\partial}{\partial x_q}\left(u_i + \left(\frac{\partial a}{\partial x_i}\right) x_i + a + v_i\right) \neq 0.
		\]
		
		Let $p \neq i,$ and $ p \neq q.$  Then the element
		\[
		\left[  d\left( \frac{\partial}{\partial x_i} \right) + \left[ \frac{\partial}{\partial x_i}, w \right] , \ x_p \frac{\partial}{\partial x_q}\right]
		\]
		has a nonzero \( \frac{\partial}{\partial x_i} \)-component. We got a contradiction. Thus, the element $a$ is a polynomial, and $w\in \widetilde{W}_X.$

		Now, considering the derivation $d - \hat{w}$ instead of $d$, we assume that $d(\mathfrak{sl}_\infty(\mathbb{F})) = (0).$ Let
		\[
		d\left(\frac{\partial}{\partial x_i}\right) = \sum_{j=1}^{\infty} f_j \frac{\partial}{\partial x_j}, \] where each $ f_j $ is a polynomial in $X.$ Arguing as above, we see that for $p \neq i$ and $q \neq p$:
		\begin{equation}\label{eq5}
			\left[ \sum_{j=1}^{\infty} f_j  \frac{\partial}{\partial x_j}, \ x_p \frac{\partial}{\partial x_q} \right] = 
			f_p \frac{\partial}{\partial x_q} - x_p  \sum_{j=1}^{\infty} \ \frac{\partial f_j}{\partial x_q} \ \frac{\partial}{\partial x_j} = 0 .
		\end{equation}
		Fix $j \geq 1$. Choosing $q \neq j$, we get $\frac{\partial f_j}{\partial x_q} = 0.$  Hence,
		$f_j = f(x_j)$ is a polynomial in $x_j.$
		
		Now the right-hand side of \eqref{eq5} looks as
		\[
		f_p  \frac{\partial}{\partial x_q} - x_p  \frac{\partial f_p}{\partial x_q}\  \frac{\partial}{\partial x_p} = 0,
		\]
		which implies
		\[
		f_p(x_p) = x_p  \frac{\partial f_q(x_q)}{\partial x_q}.
		\] It implies that $f_q(x_q) = \alpha_q + \beta_q x_q$, for $\alpha_q, \beta_q \in \mathbb{F},$ and $f_p(x_p) =  \beta_q x_p.$  Hence, there exists $\beta \in \mathbb{F}$ such that for every $p\not= i, $  we have 
		$f_p = \beta x_p.$ Choosing $q = i$, we get $f_i = \alpha + \beta x_i$. Thus,
		\[
		d\left( \frac{\partial}{\partial x_i} \right) = \alpha_i  \frac{\partial}{\partial x_i} + \beta_i  \sum_{j=1}^{\infty} x_j  \frac{\partial}{\partial x_j}, \quad \text{where} \quad \alpha_i, \ \beta_i \in \mathbb{F}.
		\]
		
		Let $1 \leq i \neq j < \infty$. Applying the derivation $d$ to $[\frac{\partial}{\partial x_i}, \frac{\partial}{\partial x_j}] = 0$, we get:
		\[
		\left[ \alpha_i   \frac{\partial}{\partial x_i} + \beta_i  \sum_{k=1}^{\infty}  x_k  \frac{\partial}{\partial x_k}, \ \frac{\partial}{\partial x_j} \right] + \left[ \frac{\partial}{\partial x_i}, \ \alpha_j  \frac{\partial}{\partial x_j} + \beta_j  \sum_{k=1}^{\infty}  x_k  \frac{\partial}{\partial x_k} \right] = 
		\]
		\[ \beta_j  \frac{\partial}{\partial x_i} -\beta_i \frac{\partial}{\partial x_j} = 0,
		\]
		which implies $\beta_i = \beta_j = 0$.
		Hence, $$d\left(\frac{\partial}{\partial x_i}\right) = \alpha_i   \frac{\partial}{\partial x_i} \quad \text{for every} \quad  i \geq 1.$$

		Applying the derivation $d$ to
		\[
		\left[ \frac{\partial}{\partial x_i}, \ x_i  \frac{\partial}{\partial x_q} \right] = \frac{\partial}{\partial x_q}, \quad i \neq q,
		\] we get: $$ \alpha_i   \frac{\partial}{\partial x_q}= \alpha_q   \frac{\partial}{\partial x_q}.$$ Hence $ \alpha_i =\alpha$  for all $ i\geq 1. $ Now
		\[
		d\left( \frac{\partial}{\partial x_i} \right) = \alpha  \frac{\partial}{\partial x_i} = \left[ \frac{\partial}{\partial x_i}, \ \alpha \sum_{k=1}^{\infty}  x_k  \frac{\partial}{\partial x_k} \right].
		\]

		Replace \( d \) with $  d - \widehat{u} ,$ where $$ u= \alpha \sum_{k=1}^\infty x_k \frac{\partial}{\partial x_k}.$$
		Since
		\[
		\left[\mathfrak{sl}_\infty(\mathbb{F}), \ \sum_{k=1}^\infty x_k \frac{\partial}{\partial x_k}\right] = (0),
		\]
		we still have \( d(\mathfrak{sl}_\infty(\mathbb{F})) = (0) \),  and in addition to that,
		\[
		d\left( \frac{\partial}{\partial x_i} \right) = 0 \quad \text{for any } i \geq 1.
		\]
		
		We have:
		\[
		\mathfrak{gl}_\infty(\mathbb{F}) = \sum_{i=1}^\infty \mathbb{F} x_i \frac{\partial}{\partial x_i} + \mathfrak{sl}_\infty(\mathbb{F}).
		\] Let us prove that \[
		d\left( x_i \frac{\partial}{\partial x_i}\right) = 0.
		\] 
		
		For any \( p \geq 1 \), we compute
		\[
		\left[ \frac{\partial}{\partial x_p}, \ x_i \frac{\partial}{\partial x_i} \right] =
		\begin{cases}
			\frac{\partial}{\partial x_i}, & \text{if } p = i, \\
			0, & \text{if } p \neq i.
		\end{cases}
		\] In both cases,
		\[
		d\left( \left[ \frac{\partial}{\partial x_p}, \ x_i \frac{\partial}{\partial x_i} \right] \right) = 0.
		\]
		Since $$ d\left( \frac{\partial}{\partial x_p} \right) = 0,  \quad \text{then} \quad 
		\left[ \frac{\partial}{\partial x_p},\  d \left( x_i \frac{\partial}{\partial x_i} \right) \right] = 0.$$
		Assume
		\[
		d \left( x_i \frac{\partial}{\partial x_i} \right) = \sum_{k=1}^{\infty} a_k \frac{\partial}{\partial x_k}.
		\]
		For any \( k \geq 1 \), we have \( \frac{\partial}{\partial x_p} a_k = 0 \). Hence, \( a_k \in \mathbb{F} \). From the fact that 
		\[
		d \left( x_i \frac{\partial}{\partial x_i} - x_j \frac{\partial}{\partial x_j} \right) = 0,
		\]
		it follows that the element 
		\[
		\sum_{k = 1}^{\infty} a_k \frac{\partial}{\partial x_k}
		\]
		does not depend on the choice of the index  \( i \).
		
		Let \( q \ne i \). Then
		\[
		\left[ x_i \frac{\partial}{\partial x_i}, \ x_i \frac{\partial}{\partial x_q} \right] = x_i \frac{\partial}{\partial x_q}.
		\] Applying the derivation \( d \), we obtain:
		\[
		\left[ \sum_{k=1}^\infty a_k \frac{\partial}{\partial x_k}, \ x_i \frac{\partial}{\partial x_q} \right] = 0.
		\] Consequently, \( a_i = 0 \), which completes the proof that
		\[
		d(\mathfrak{gl}_\infty(\mathbb{F})) = (0).
		\]
		
		Thus, \[
		d(a) = \left[a, \ \alpha  \sum_{k=1}^\infty  x_k  \frac{\partial}{\partial x_k}\right] = 0 \quad \text{for all} \quad a \in \mathfrak{gl}_\infty(\mathbb{F}).
		\] The element $$\alpha  \sum_{k=1}^\infty x_k  \frac{\partial}{\partial x_k}$$ lies in $\widetilde{W}_{\text{fin}}.$ Thus, we proved that $d$ is an inner derivation. This completes the proof of the lemma. 
	\end{proof}

	\begin{Lemma}\label{Lemma9} If $w \in \widetilde{W}_X$ and $\mathfrak{sl}_\infty(\mathbb{F}) w \subseteq \widetilde{W}_{\text{fin}}$, then $w \in \widetilde{W}_{\text{fin}}$. \end{Lemma} 
	
	\begin{proof} Let $$w = \sum_{j=1}^{\infty} f_j \frac{\partial}{\partial x_j}\in \widetilde{W}_X.$$ Suppose $p \neq q$. By \eqref{eq5}, we get:
		\[
		w' = f_p \frac{\partial}{\partial x_q}- x_p \sum_{j=1}^{\infty} \left( \frac{\partial f_j}{\partial x_q}\right) \frac{\partial}{\partial x_j} \in \widetilde{W}_{\text{fin}}.
		\] If the variable $x_q$ is involved in $f_j$ for $j \neq q$, then the variable $x_p$ is involved in the $\frac{\partial}{\partial x_j}$-component of $w'$. Since $w' \in \widetilde{W}_{\text{fin}}$, it follows that $x_q$ is involved in finitely many $f_j$'s. This completes the proof of the lemma. \end{proof}
	
	It remains to prove the following lemma.
	\begin{Lemma}\label{Lemma10}  Let $d : W_\infty(\mathbb{F}) \to \widetilde{W}$ be a derivation such that
		\[
		d\left( \sum_{i=1}^{\infty} \mathbb{F}  \frac{\partial}{\partial x_i} + \mathfrak{gl}_\infty(\mathbb{F}) \right) = (0).
		\]
		Then \( d = 0 \). \end{Lemma}
	
	\begin{proof} Let \( m \) be a monomial. We will use induction on the length of the monomial \( m \) and show that
		\[
		d\left( m \frac{\partial}{\partial x_i} \right) = 0.
		\]  
		If the length is $0$ or $1,$ then everything clearly follows from the assumption. Suppose that  the length of \( m \) be greater than or equal to  \(  2 \).  
		Let
		\[
		d\left( m \frac{\partial}{\partial x_i} \right) = \sum_{j=1}^\infty f_j \frac{\partial}{\partial x_j}.
		\]  
		By the induction assumption, for any \( k \geq 1 \), we have:
		\[
		d\left( \Big[m \frac{\partial}{\partial x_i}, \  \frac{\partial}{\partial x_k}\Big] \right) 
		= \sum_{j=1}^\infty   \Big( \frac{\partial f_j}{\partial x_k} \Big) \frac{\partial}{\partial x_j}  = 0.
		\]  
		This implies that \( f_j \in \mathbb{F} \) for any \( j \geq 1 \).
		
		Now let \( p \neq i \). Choose \( q \geq 1 \) such that \( p \neq q \) and  \(  x_q \) does not appear  in the monomial \( m \). Then we have that 
		\[
		\left[m \frac{\partial}{\partial x_i}, \ x_p \frac{\partial}{\partial x_q}\right] = 0,
		\]
		 which implies the following equality:
		\begin{align*}
			d\left( \Big[m \frac{\partial}{\partial x_i}, x_p \frac{\partial}{\partial x_q}\Big] \right) 
			&= \Big[d\Big(m \frac{\partial}{\partial x_i}\Big),  x_p \frac{\partial}{\partial x_q}\Big] + \Big[m \frac{\partial}{\partial x_i},  d\Big(x_p \frac{\partial}{\partial x_q}\Big)\Big] = 0.
		\end{align*}
		
		Moreover, 
		\[ x_p \frac{\partial}{\partial x_q} \in \mathfrak{sl}_\infty(\mathbb{F}),
		\quad \text{so} \quad d\left(x_p \frac{\partial}{\partial x_q}\right) = 0.
		\]
		Consequently, 
		\[
		\left[d\left(m \frac{\partial}{\partial x_i}\right), \ x_p \frac{\partial}{\partial x_q}\right] = \left[ \sum_{j=1}^{\infty} f_j \frac{\partial}{\partial x_j}, \ x_p \frac{\partial}{\partial x_q} \right] =  f_p \frac{\partial}{\partial x_q}=0.
		\]
		It follows that \( f_q = 0 \) for all $p\not= i$; and  
		\[
		d\left( m \frac{\partial}{\partial x_i} \right) = \alpha \frac{\partial}{\partial x_i}, \quad \alpha \in \mathbb{F}.
		\]
		
		First, suppose that \( x_i \) is involved in \( m \); that is, \( m = x_i m' \).  
		Choose \( j \) such that \( x_j \) is not involved in \( m \). Then:
		\[
		\left[m' x_j \frac{\partial}{\partial x_i}, \ x_i \frac{\partial}{\partial x_j}\right] 
		= m' x_j \frac{\partial}{\partial x_j} - x_i m' \frac{\partial}{\partial x_i}.
		\]
		Applying the derivation \( d \) to both sides, we obtain:
		\[
		\left[d\Big(m' x_j \frac{\partial}{\partial x_i}\Big), \ x_i \frac{\partial}{\partial x_j}\right]
		= d\Big(m' x_j \frac{\partial}{\partial x_j}\Big) - d\Big( m \frac{\partial}{\partial x_i}\Big).
		\]
		The left-hand side lies in \( \mathbb{F}  \frac{\partial}{\partial x_j} \), 
		as does the right-hand side: $d\Big(m' x_j \frac{\partial}{\partial x_j}\Big) $ lies in $ \mathbb{F}  \frac{\partial}{\partial x_j} $, whereas 
		$d\Big( m \frac{\partial}{\partial x_i}\Big)$ lies in $\mathbb{F}  \frac{\partial}{\partial x_i}.$  Hence, $
		d\left(m \frac{\partial}{\partial x_i}\right) = 0.$
		
		Now suppose that \( x_i \) is not involved in \( m \).  We have: $ m = m' x_k$ for $k\not= i.$ 
		Then:
		\[
		\left[x_k \frac{\partial}{\partial x_i},\ m' x_i \frac{\partial}{\partial x_i}\right]
		= m \frac{\partial}{\partial x_i}.
		\]
		By the above,
		\[
		d\left(m' x_i \frac{\partial}{\partial x_i}\right)= 0, \quad \text{which implies that:} \quad 
		d\left(m \frac{\partial}{\partial x_i}\right) = 0.
		\]
		Thus,  \( d = 0 \). This completes the proof of the lemma. \end{proof}

	\begin{proof}[Proof of Theorem \ref{Theo1}]
		Let \( d: W_\infty(\mathbb{F}) \to W_\infty(\mathbb{F}) \) be a derivation.  
		By Lemmas \ref{Lemma8} and \ref{Lemma9}, there exists an element \( w \in W_{\mathrm{fin}}(\mathbb{F}) \) such that 
		$d(a) = [a, w]$  for any element $$ a \in \sum_{i=1}^\infty \mathbb{F}  \frac{\partial}{\partial x_i} + \mathfrak{gl}_\infty(\mathbb{F}).$$ Now by Lemma \ref{Lemma10}, we conclude that \( d = \hat{w} \).  
		This completes the proof of Theorem \ref{Theo1}. \end{proof}

	\section{The poof of Theorem \ref{Theo2}}
	
	\begin{Lemma}\label{Lemma11} The centralizer of \( W_\infty(\mathbb{F}) \) in \( \widetilde{W} \) is zero. \end{Lemma}
	
	\begin{proof} Suppose that \( w \in \widetilde{W} \) and $[W_\infty(\mathbb{F}), w] = (0).$ 
		In particular,
		\[
		\left[ \frac{\partial}{\partial x_i},\ w\right] = 0 \quad \text{for every } i \geq 1.
		\]
		This implies that 
		\[
		w = \sum_{i=1}^\infty \alpha_i \frac{\partial}{\partial x_i}, \quad \text{where} \quad \alpha_i \in \mathbb{F}.
		\]
		Let us assume that \( \alpha_i \neq 0 \). Then
		\[
		[w, x_i \frac{\partial}{\partial x_1}] = \alpha_i \frac{\partial}{\partial x_1}  \neq 0.
		\]
		This contradicts the assumption and completes the proof of the lemma. \end{proof}

	\begin{proof}[Proof of Theorem \ref{Theo2}] Let \( d : W_{\mathrm{fin}}(\mathbb{F}) \to W_{\mathrm{fin}}(\mathbb{F}) \) be a derivation. Arguing as in the proof of Theorem \ref{Theo1}, we find an element \( w \in W_{\text{fin}}(\mathbb{F}) \) such that  
		\[
		(d - \hat{w})(W_{\infty}(\mathbb{F})) = (0).
		\]
		Since \( W_{\infty}(\mathbb{F}) \) is an ideal in \( W_{\text{fin}}(\mathbb{F}) \), we conclude that
		\[
		(d - \hat{w})\Big([W_{\text{fin}}(\mathbb{F}), W_{\infty}(\mathbb{F})]\Big) = \Big[(d - \hat{w}) \big( W_{\text{fin}}(\mathbb{F})\big), \ W_{\infty}(\mathbb{F}) \Big]= (0).
		\]
		By Lemma \ref{Lemma11},  we have that \( d = \hat{w} \), which completes the proof of  Theorem \ref{Theo2}. \end{proof}

	\section{The poof of Theorem \ref{Theo3}}
	
	To prove Theorem \ref{Theo3}, we will need a couple of lemmas. 
	
	Arguing as above, we assume that $d:W_X (\mathbb{F}) \to W_X (\mathbb{F})$ is a derivation and $d(W_{\infty}(\mathbb{F}))=(0).$
	\begin{Lemma}\label{Lemma12} Let \( i_0 > 1 \), 
		\[ \text{and let} \quad 
		a \in \sum_{j \neq i_0} \mathbb{F}[X \setminus \{x_{i_0}\}] \frac{\partial}{\partial x_j}. \quad \text{Then} \quad   d(a) \in \sum_{j \neq i_0} \mathbb{F}[X \setminus \{x_{i_0}\}] \frac{\partial}{\partial x_j} .\]
	\end{Lemma}
	
	\begin{proof} Let 
		\[
		d(a) = \sum_{j=1}^{\infty} f_j \frac{\partial}{\partial x_j}.
		\]
		From \( [a, \frac{\partial}{\partial x_{i_0}}] = 0 \), it follows that
		\[
		\left[d(a), \ \frac{\partial}{\partial x_{i_0}}\right] = 0. \] Hence, $ f_j \in \mathbb{F}\left[X \setminus \{x_{i_0}\}\right]$ for all $ j \geq 1.$
		Similarly, 
		\[
		\left[d(a), \ x_{i_0} \frac{\partial}{\partial x_{i_0}}\right] = 0 .\] Thus, $f_{i_0} = 0,$ and the lemma is proved.  \end{proof}

	\begin{Lemma}\label{Lemma13} Let \[ a = \sum_{j \neq i_0} f_j \frac{\partial}{\partial x_j} \in W_X(\mathbb{F}), \] where all polynomials  \( f_j \) are divisible by \( x_{i_0}^k \) for some \( k \geq 1 \). Then 
		\[
		d(a) \in x_{i_0}^{k-1} \ W_X(\mathbb{F}).
		\] \end{Lemma}
	
	\begin{proof} Let \( f_j = x_{i_0}^k f'_j \), where \( f_j' \in \mathbb{F}[X] .\) Choose polynomials \( \widetilde{f_j} \in \mathbb{F}[X] \)  such that 
		$$\frac{\partial}{\partial x_{i_0}} \widetilde{f_j} = f_j' .$$ Then 
		\[
		a = \left[x_{i_0}^k \frac{\partial}{\partial x_{i_0}}  , \ \sum_{j \neq i_0} \widetilde{f_j} \frac{\partial}{\partial x_j}\right].
		\]
		This implies that
		\[
		d(a) \in  \left[x_{i_0}^k \frac{\partial}{\partial x_{i_0}}  , \ W_X(\mathbb{F})\right] \subseteq x_{i_0}^{k-1} \ W_X(\mathbb{F}),
		\]
		and completes the proof of the lemma. \end{proof}

	\begin{proof}[Proof of Theorem \ref{Theo3}] Suppose that $w \in W_X(\mathbb{F})$, and  \[ d(w) = \sum_{i=1}^{\infty} h_i \frac{\partial}{\partial x_i}  \neq 0 .\] Choose \( i_0 \geq 1 \) such that \( h_{i_0} \neq 0 \). Let \( l \geq 1 \) be such that \( h_{i_0} \notin x_{i_o}^l \mathbb{F}[X] \).
		
		Let 
		\[
		w = \sum_{i=0}^\infty f_i \frac{\partial}{\partial x_i}.
		\]
		Since \( d(W_{\infty}(\mathbb{F})) = (0) ,\) we may assume without loss of generality that \( f_{i_0} = 0 \). Then the element \( w \) can be represented as
		\[
		w = w_0 + x_{i_0} w_1 +  \cdots + x_{i_0}^l w_l + w', \quad \text{where}
		\]
		\[  w_0, \ w_1, \ \dots, \ w_l \in \sum_{j \neq i_0} \mathbb{F}[X \setminus \{x_{i_0}\}]  \frac{\partial}{\partial x_j}, \quad \text{and} \quad w' \in x_{i_0}^{l+1} \ W_X(\mathbb{F}) .\]
		
		Choose  $i_1 \neq i_0.$ Now let 
		\[
		w_k = \sum_{j \neq i_0} f_{k j}  \frac{\partial}{\partial x_j} \quad \text{for} \quad 0 \leq k \leq l.
		\] Choose polynomials $$\widetilde{f_{kj}} \in \mathbb{F}[X \setminus \{x_{i_0}\}] \quad \text{such that} \quad  \frac{\partial \widetilde{f_{kj}}}{\partial x_{i_1}}=f_{kj}.$$ Consider $$ 
		\widetilde{w_k} = \sum_{j \neq i_0} \widetilde{f_{k j}}  \frac{\partial}{\partial x_j} .$$ We get
		\[
		\left[ x_{i_0}^k \frac{\partial}{\partial x_{i_1}}, \ \widetilde{w_k} \right] = x_{i_0}^k w_k,\] which implies that \[  d( x_{i_0}^k\, w_k )= \left[d\left( x_{i_0}^k \frac{\partial}{\partial x_{i_1}} \right),  \widetilde{w_k} \right] +\left[ x_{i_0}^k \frac{\partial}{\partial x_{i_1}},  d(\widetilde{w_k}) \right]=  \left[ x_{i_0}^k \frac{\partial}{\partial x_{i_1}},  d(\widetilde{w_k}) \right].
		\]
		Hence, the \( \frac{\partial}{\partial x_{i_0}} \)-component of \( d\left( x_{i_0}^k w_k\right) \) is zero by Lemma \ref{Lemma12}. Since $d(w)=d(w_0)+d(x_{i_0}w_1)+\cdots + d(x_{i_0}^l w_l)+ d(w'),$ we get that the \( \frac{\partial}{\partial x_{i_0}} \)-component of \( d(w) \) is equal to the \( \frac{\partial}{\partial x_{i_0}} \)-component of \( d(w') \). By Lemma~\ref{Lemma13}, this $  \frac{\partial}{\partial x_{i_0}}$-component lies in $$ x_{i_0}^l\mathbb{F}[X] \frac{\partial}{\partial x_{i_0}} ,$$ a contradiction. Thus, we have shown that \( d = 0 \), which completes the proof of Theorem \ref{Theo3}. \end{proof}
	
	\section{Some Conjectures}

	A.~Rudakov \cite{Rudakov1} (see also V.~Bavula \cite{Bavula,Bavula2,Bavula3}, H.-P.~Kraft and A.~Re\-ge\-ta~\cite{Kraft}) showed that every automorphism of the Lie algebra $W_{n}(\mathbb{F})$ is a conjugation by an automorphism of $\mathbb{F}[X_n]$. Consider the following group of automorphisms of the polynomial algebra $\mathbb{F}[X]$: \\
	$\mathrm{Aut}_{\mathrm{fin}}\, \mathbb{F}[X] = \big\{ \varphi \in \mathrm{Aut}\, \mathbb{F}[X] \  \big| $   every variable $x_n $     occurs in finitely many polynomials $\varphi(x_i)$ for $i \geq 1,$   and in finitely many polynomials 	$\varphi^{-1}(x_i)$ for $ i \geq 1 \}.$

\begin{Conjecture} Every automorphism of the Lie algebra $W_{\infty}(\mathbb{F})$ is a conjugation by an automorphism $\varphi \in \mathrm{Aut}_{\mathrm{fin}}\, \mathbb{F}[X]$. \end{Conjecture}

\begin{Conjecture}  Every automorphism of the Lie algebra $W_{X}(\mathbb{F})$ is a conjugation by an automorphism from $\mathrm{Aut}\, \mathbb{F}[X]$. \end{Conjecture}


\begin{thebibliography}{99}
		
		\bibitem{Bavula} Bavula V.V., The group of automorphisms of the Lie algebra of derivations of a polynomial algebra, J. Algebra Appl. \textbf{16}(05), 1750088 (2017).
		
		\bibitem{Bavula2} Bavula V.V., The group of automorphisms of the Lie algebra of derivations of a field of rational functions, Glasg. Math. J. \textbf{59}(3) (2017), P.513-524.
		
		\bibitem{Bavula3} Bavula V.V., The groups of automorphisms of the Lie algebras of triangular  polynomial derivations, J. Pure Appl. Algebra \textbf{218}(5) (2014), P.829-851.
		
		\bibitem{Bei_Bre_Cheb_Mart_1} Beidar K.I., Bre\v{s}ar M., Chebotar M.A.,  Martindale 3rd W.S.,  On Herstein's Lie map conjectures I, Trans. Amer. Math. Soc. \textbf{353} (2001), P.4235-4260.
		
		\bibitem{Bei_Bre_Cheb_Mart_2} Beidar K.I., Bre\v{s}ar M., Chebotar M.A.,  Martindale 3rd W.S.,  On Herstein's Lie map conjectures II, J. Algebra \textbf{238} (2001), P.239-264.
		
		\bibitem{Bei_Bre_Cheb_Mart_3} Beidar K.I., Bre\v{s}ar M., Chebotar M.A.,  Martindale 3rd W.S.,  On Herstein's Lie map conjectures III, J. Algebra \textbf{249} (2002), P.59-94.
		
		
		
		\bibitem{Bezushchak1} Bezushchak O., Derivations and automorphisms of locally matrix algebras, J.~Algebra  \textbf{576} (2021), P.1-26. 
		
		
		
		
		\bibitem{Bezushchak2} Bezushchak O., Automorphism and derivations of algebras of infinite matrix. Linear Algebra App. \textbf{650}(2) (2022), P.42-59.
		

		
		\bibitem{Bezushchak_Kashuba_Zelmanov} Bezushchak O., Kashuba I., Zelmanov E., On Lie isomorphisms of rings, Mediterr. J. Math. (2025) 22:80.
		
		\bibitem{BezOl_2} Bezushchak O., Oliynyk B., Primary decompositions of unital locally matrix algebras, Bull. Math. Sci. \textbf{10}(1) (2020), P.2050006-1 - 2050006-7.
		
		
		\bibitem{Bezushchak_Petravchuk_Zelmanov} Bezushchak O., Petravchuk A., Zelmanov E., Automorphisms and derivations of affine commutative and $PI$-algebras, Trans. Amer. Math. Soc. \textbf{377}(2) (2024), P.1335-1356 (2024).  
		
		\bibitem{Doković_Zhao}   Doković D.Ž.,  Zhao K., Derivations, Isomorphisms, and Second Cohomology of Generalized Witt Algebras, Trans. Amer. Math. Soc. \textbf{350}(2) (1998), P.643-664.
		
	\bibitem{Farnsteiner} Farnsteiner R., On the cohomology of associative algebras and Lie algebras,
	Proc. Amer. Math. Soc.  \textbf{99} (1987), P.415-420.
		
		\bibitem{Jacobson} Jacobson~N., Lectures in Abstract Algebra: \textbf{II}. Linear Algebra (Graduate Texts in Mathematics), Springer-Verlag, Berlin-Heidelberg-New York, 1975. 
		
		\bibitem{Jacobson_PBW} Jacobson N., Lie Algebras. Interscience, New York, 1962.
		
		
		
		\bibitem{Klymenko_Lysenko_Petravchuk}  Klymenko I.S.,  Lysenko S.V.,  Petravchuk A.P., Lie algebras of derivations with large abelian ideals, Algebra Discrete Math. \textbf{28}(1) (2019), P.123-129.
		
	
		

		
				
		
	\bibitem{Kraft} 	Kraft H., Regeta A., Automorphisms of the Lie algebra of vector fields, J. Eur. Math. Soc. \textbf{19} (2017), P.1577-1588.
		
		\bibitem{Makedonskyi_Petravchuk} Makedonskyi Ie.O., Petravchuk A.P., On nilpotent and solvable Lie algebras of derivations, J. Algebra, \textbf{401} (2014), P.245-257.
		
			\bibitem{Morimoto} Morimoto T., The derivation algebras of the classical infinite Lie algebras, Math.	Kyoto Univ. \textbf{16} (1976), P.1-24.
		
		\bibitem{Neeb} Neeb K.-H., Derivations of locally simple Lie algebras, J. Lie Theory \textbf{15} (2005), P.589-594.
		
		

		
		
		\bibitem{Rudakov1}  Rudakov A.N., Groups of automorphism of infinite-dimensional simple Lie algebras, Mathematics of the USSR-Izvestiya \textbf{3}(4) (1969), 748-764.    
		
		
		


		\bibitem{Strade} Strade H., Simple Lie Algebras over Fields of Positive Characteristic, I. Structure Theory, de Gruyter Expositions in Mathematics \textbf{38}, Berlin-New York, 2004.
		
\bibitem{Takens}	Takens F., Derivations of vector fields, Compositio Mathematica \textbf{26} (1973), P.151-158.


	
	\end{thebibliography}
\end{document}